\documentclass{amsart}
\renewcommand{\baselinestretch}{1.5}
\usepackage{graphicx,float}
\usepackage[mathscr]{eucal}
\usepackage{amsfonts}
\usepackage{enumerate}
\usepackage{multicol}
\usepackage{array}
\usepackage{tikz}
\usepackage{mytab}
\usepackage{stackengine}
\usepackage{subfig,caption}

\theoremstyle{plain}
\newtheorem{thm}{Donotwrite}[section]

\newtheorem{theorem}[thm]{Theorem}

\newtheorem{lemma}[thm]{Lemma}

\theoremstyle{definition}
\newtheorem{example}[thm]{Example}

\newtheorem{remark}[thm]{Remark}

\numberwithin{equation}{section}

\newcommand{\nc}{\newcommand}

\nc{\op}{\oplus} \nc{\pv}{P^{\vee}}

\newcommand{\Z}{\mathbb{Z}}
\newcommand{\C}{\mathbb{C}}

\newcommand{\La}{\Lambda}
\newcommand{\la}{\lambda}
\newcommand{\al}{\alpha}

\newcommand{\g}{\mathfrak{g}}

\nc{\B}{\mathbf{B}} \nc{\V}{\mathbf{V}} 
\nc{\nbinom}[2]{\genfrac{}{}{0pt}{1}{{#1}}{{#2}}}
\nc{\qbinom}[2]{\left[\genfrac{}{}{0pt}{1}{{#1}}{{#2}}\right]}
\nc{\ft}{\tilde{f}} 
\nc{\et}{\tilde{e}} 
\nc{\Y}{\mathbf{Y}}
\nc{\T}{\mathbf{T}}
\nc{\R}{\mathbf{R}}
\nc{\D}{\mathbf{D}}
\nc{\ra}{\rightarrow} 
\nc{\vep}{\varepsilon} 
\nc{\vp}{\varphi}
\nc{\h}{\mathfrak{h}} 
\nc{\oP}{\overline{P}}
\nc{\Fit}{\tilde{F}_i} 
\nc{\Eit}{\tilde{E}_i}
\nc{\fit}{\tilde{f}_i} 
\nc{\eit}{\tilde{e}_i}
\nc{\mf}{\mathfrak}
\nc{\ds}{\displaystyle}
\nc{\oa}{a'}
\nc{\ob}{b'}
\nc{\mc}{\mathcal}
\nc{\bsx}{\boldsymbol{x}}
\nc{\imin}{i_{min}}
\nc{\imax}{i_{max}}

\nc{\dmin}{d_{min}}
\nc{\dmax}{d_{max}}

\allowdisplaybreaks

\begin{document}

\title{Multiplicities of maximal weights of the $\widehat{s\ell}(n) $-module $V(k\Lambda_0)$}
\author{Rebecca L. Jayne}
\address{Hampden-Sydney College, Hampden-Sydney, VA 23943}
\email{rjayne@hsc.edu} 

\author{Kailash C. Misra}
\address{Department of Mathematics, North Carolina State University,  Raleigh,  NC 27695-8205}
\email{misra@ncsu.edu}

\thanks{KCM: partially supported by Simons Foundation grant \#  636482}

\begin{abstract} Consider the affine Lie algebra $\widehat{s\ell}(n)$ with null root $\delta$, weight lattice $P$ and set of dominant weights $P^+$. Let $V(k\Lambda_0), \, k \in \mathbb{Z}_{\geq 1}$ denote the integrable highest weight $\widehat{s\ell}(n)$-module with level $k \geq 1$ highest weight $k\Lambda_0$. Let $wt(V)$ denote the set of weights of $V(k\Lambda_0)$. A weight $\mu \in wt(V)$ is a maximal weight if 
$\mu + \delta  \not\in wt(V)$. Let $max^+(k\La_0)= max(k\Lambda_0)\cap P^+$  denote the set of maximal dominant weights which is known to be a finite set. In 2014, the authors gave the complete description of the set $max^+(k\Lambda_0)$  \cite{JM1} . In subsequent papers \cite{JM2, JM3}, the multiplicities of certain subsets of $max^+(k\Lambda_0)$ were given in terms of some pattern-avoiding permutations using the  associated crystal base theory.  In this paper the multiplicity of all the maximal dominant weights of the $\widehat{s\ell}(n) $-module $V(k\Lambda_0)$ are given generalizing the results in \cite{JM2, JM3}.
\end{abstract}

\maketitle

\section{Introduction} 
Affine Lie algebras form an important class of infinite dimensional Kac-Moody Lie algebras (cf: \cite{Kac}) generalizing finite dimensional semisimple Lie algebras. Unlike finite dimensional semisimple Lie algebras, an affine Lie algebra has a one dimensional center and not all its roots are Weyl group conjugate to some simple root. The set of roots which are Weyl group conjugate to simple roots are called real roots; hence their multiplicities are one. The roots which are not real are called imaginary roots and these roots could have multiplicity greater than one. The affine Lie algebras are classified and the multiplicities of each of their roots are known \cite{Kac}. Let $\g$ be an affine Lie algebra over the field of complex numbers $\C$ with simple roots $\{\al_i \mid 0 \leq i \leq n-1\}$, simple coroots $\{h_i \mid 0 \leq i \leq n-1\}$, null root $\delta$, canonical central element $c$, and fundamental weights $\{\Lambda_i \mid 0 \leq i \leq n-1\}$. Then the subalgebra $\h = {\rm span}\{h_0, \cdots , h_{n-1}, d\}$ is called the Cartan subalgebra where $d$ is a suitable degree derivation. The set $P = \Z\Lambda_0 \oplus \cdots \oplus \Z\Lambda_{n-1} \oplus \Z\delta$ is the weight lattice and $P^+ = \{\lambda \in P \mid \lambda(h_i) \in \Z_{\geq 0}, 0\leq i \leq n-1\}$ is the set of dominant weights. For each dominant weight $\lambda \in P^+$ there is a unique (up to isomorphism) irreducible integrable $\g$-module $V(\lambda)$ with highest weight $\lambda$. Any weight $\mu$ of $V(\lambda)$ is of the form $\mu = \lambda -\sum_{i =0}^{n-1}m_i\al_i$ with each $m_i \in \Z_{\geq 0}$. Determining the multiplicities of these weights is still an open problem. A weight $\la \in P^+$ is of level $k$ if $\la (c) = k$. A weight $\mu$ of $V(\lambda)$ is a maximal weight if $\mu + \delta$ is not a weight. Let $max(\lambda)$ denote the set of maximal weights of $V(\lambda)$. Then any weight $\mu$ of $V(\lambda)$ is of the form $\mu = \nu - l\delta$ for some $\nu \in max(\lambda)$, and $l \in \Z_{\geq 0}$. Since weight multiplicities are invariant under the action of the Weyl group of $\g$, it is sufficient to focus on the set of maximal dominant weights $max^+(\La) = max(\La)\cap P^+$. It is known \cite[Proposition 12.6]{Kac} that the set $max^+(\La)$ is finite. However, explicit description of these weights is a nontrivial task. In \cite{JM1}, the authors gave complete descriptions of maximal dominant weights of the $\widehat{s\ell}(n)$-module $V(\lambda)$ for the dominant weights $\lambda = (k-1)\Lambda_0 + \Lambda_s$, $0\leq s \leq n-1$ of level $k \in \Z_{\geq 2}$. More recently, Y-H Kim, S-J Oh and Y-T Oh \cite{KOO} have determined closed form and recursive formulae for the cardinality of the set of maximal dominant weights $max^+(\La)$  for any affine Lie algebra $\g$ and any $\Lambda \in P^+$ using the cyclic sieving phenomenon \cite{RSW}. In particular, they define an equivalence relation $\Lambda \sim \Lambda'$ for $\Lambda , \Lambda' \in P^+$ if $|max^+(\La)| = |max^+(\La ')|$ and showed that in the case of the affine Lie algebra $\widehat{s\ell}(n)$ the set of weights $\{ (k-1)\Lambda_0 + \Lambda_s$, $0\leq s \leq n-1\}$ form the representatives of distinct equivalence classes of maximal dominant weights of level $k$.

In general finding the multiplicities of the weights for the integrable representations of affine Lie algebras is an important and nontrivial problem.
In \cite{Tsu}, Tsuchioka showed that the multiplicities of the maximal dominant weights for the  $\widehat{s\ell}(n)$-module $V(2\La_0)$ are related to the Catalan numbers which are in bijection with the $321$-avoiding permutations \cite{Sta}. Motivated by this work, the authors determined the maximal dominant weights of the $\widehat{s\ell}(n)$-module $V(k\La_0)$ and conjectured that the multiplicities of a certain subclass of these weights are given by $(k+1)k \cdots 21$-avoiding permutations which they proved in \cite{JM2}. An independent proof of this conjecture was also given in \cite{TW}. In \cite{KKO}, J-S Kim, K-H Lee and S-J Oh gave multiplicities of some level $k= 2, 3$ maximal dominant weights of the $\g$-modules $V(k\La), \La \in P^+$ for $\g = B_n^{(1)}, D_n^{(1)}, A_n^{(2)}$ and $D_{n+1}^{(2)}$ using certain combinatorial objects. In \cite{JM3}, the authors gave multiplicities of a larger class of maximal dominant weights of the $\widehat{s\ell}(n)$-modules $V(k\La_0)$ in terms of pattern-avoiding permutations generalizing the results in \cite{JM2}.

In this paper we reformulate the descriptions for the maximal dominant weights of the $\widehat{s\ell}(n)$-modules $V(k\La_0)$ given in \cite{JM1} and give the multiplicities of each of these maximal dominant weights in terms of certain pattern-avoiding permutations generalizing the results in \cite{JM3}. In order to do this we use the explicit descriptions of the crystal base for the $\widehat{s\ell}(n)$-modules $V(k\La_0)$ in terms of extended Young diagrams given in \cite{JMMO}. Using the crystal base properties, first we describe the multiplicities as the number of certain ordered pair of Young tableaux. Then using the RSK correspondence \cite{Sta}, we prove that the number of ordered pairs of these Young tableaux is same as the number of certain pattern-avoiding permutations, generalizing the results in \cite{JM3}. However, the approach in this paper is different from that in \cite{JM3}.

\section{Preliminary} 
Consider the affine Lie algebra $\g = \widehat{s\ell}(n) = s\ell(n, \C) \otimes \C[t, t^{-1}] \oplus \C c \oplus \C d$ (cf: \cite{Kac}) with simple roots $\{\al_i\mid 0\leq i\leq n-1\}$, simple coroots $\{h_i\mid 0\leq i\leq n-1\}$, null root $\delta = \sum_{i=0}^{n-1}\alpha_i$ and canonical central element $c =  \sum_{i=0}^{n-1}h_i$. The set of weights $\{\La_i\mid 0\leq i\leq n-1\}$ are called dominant weights where $\La_j(h_i) = \delta_{i,j}, \La_j(d) = 0$. Then $P = \oplus_{j=0}^{n-1}\Z\La_j\oplus \Z\delta$  is the weight lattice and $P^+ = \{\la \in P \mid \la(h_i) \geq 0, i= 0, 1,\cdots , n-1\}$ is the set of dominant weights. For each dominant weight $\La \in P^+$ there is a unique (up to isomorphism) irreducible integrable $\g$-module $V(\La)$ of level $k = \La(c)$ with highest weight $\La$. Let $wt(\La)$ denote the set of weights of $V(\La)$. Recall that a weight $\la \in wt(\La)$ is a maximal weight if $\la + \delta \not\in wt(\La)$. Let $max(\La)$ denote the set of maximal weights of $V(\La)$. Then for $\la \in wt(\La)$  it is known that $\la = \mu - r\delta$ for some $\mu \in max(\La) , r \in \Z_\geq 0$ \cite[Equation 12.6.1]{Kac}. Knowing the explicit forms of the maximal weights and their multiplicities are important in understanding the structure of $V(\La)$. Since weight multiplicities are invariant under the action of the Weyl group of $\g$, it is sufficient to focus on the set of maximal dominant weights $max(\La)\cap P^+$. It is known \cite[Proposition 12.6]{Kac} that the set $max(\La)\cap P^+$ is finite. For $k \geq 1$, consider the irreducible integrable $\g$-module $V(k\Lambda_0)$ with highest weight $\Lambda = k\Lambda_0$. The set of maximal dominant weights of $V(k\Lambda_0)$ are given in \cite{JM1}. We can reformulate the expressions for these maximal dominant weights using partitions as follows. 

We say a pair of  partitions $(A =\{a_1 \geq a_2 \geq \cdots \geq a_r \geq 1\}$,  $B =\{b_1 \geq b_2 \geq \cdots \geq b_q \geq 1\})$ of $\ell \in \Z_{\geq 1}$ is an admissible pair if $a_1 + b_1 \leq k$ and $r+q \leq n$. Define
$\ds \lambda_{A,B}^{\ell} = \ell \alpha_0 + \sum_{i=1}^{q-1} (\sum_{j=i+1}^q b_j)\alpha_i +  \sum_{i=1}^{r-1}(\sum_{j=i+1}^r a_j)\alpha_{n-i}.$
Observe that since $A$ and $B$ are partitions of $\ell$, $ \frac{\ell}{a_1} \leq r$, $\frac{\ell}{b_1} \leq q$ and so $\lceil \frac{\ell}{b_1} \rceil  +  \lceil \frac{\ell}{a_1} \rceil  \leq n$. Hence it follows from \cite[Lemma 3.1]{JM1} that
$\ds \max\{a_1,b_1\} \leq \ell \leq$ $ \max \left \{\frac{b_1(a_1 n - m)}{a_1+b_1}, \frac{b_1(a_1 n - m)}{a_1+b_1} + m - a_1 \right \}$ with $b_1n \equiv m \pmod{a_1 + b_1}$. 
Then by \cite [Theorem 3.6]{JM1}, the set of maximal dominant weights is
$max(k\Lambda_0) \cap P^+ = \{k\Lambda_0 - \lambda_{A,B}^\ell \mid \ell \geq 1, (A, B) \, {\rm admissible} \, {\rm pairs}\}$.

In this paper, we show that the multiplicity of each of these maximal dominant weights can be given by the number of certain pattern-avoiding permutations of $[\ell] = \{1, 2, 3, \ldots, \ell\}$.  These multiplicity results agree with \cite[Corollary 3.3]{JM2}  in the special case of admissible  partition pairs of $\ell$: $A = (1, 1, 1, \ldots, 1) =  B$ and \cite[Theorem 4.4]{JM3}  in the special case of admissible partition pairs of $\ell$: $A = (a, 1, 1, \ldots, 1) , B = (b, 1, 1, \ldots, 1)$.

\section{Crystal Base of $V(k\Lambda_0)$}
We briefly describe the extended Young diagram realization of the crystal $B(k\Lambda_0)$ for the $\g$-module $V(k\Lambda_0)$ given in \cite{JMMO} which we will use to determine the desired weight multiplicities.   

 An extended Young diagram $Y$ of charge $0$ is a weakly increasing sequence $Y = (y_{i})_{i \geq 0}$ with integer entries such that $y_{i} = 0$ for $i  \gg  0$.  
 For each element $y_{i}$ of the sequence, we draw a column with depth $-y_{i}$, aligned so the top of the column is on the line $y = 0$ on the right half plane.  We fill in square boxes for all columns from the depth to the line $y=0$ and obtain a diagram with a finite number of boxes.  We color a box with lower right corner at $(a,b)$ by color $j$, where $(a+b) \equiv j \pmod{n}$.   For simplicity, we refer to color $(n-j)$ by $-j$.   Thus an extended Young diagram $Y$ of charge $0$ is uniquely determined by a (finite) colored Young diagram with the top left box of color $0$. The weight of an extended Young diagram of charge $0$ is $ wt(Y) = \Lambda_{0} - \sum_{j=0}^{n-1}c_{j}\alpha_{j},$ where $c_{j}$ is the number of boxes of color $j$ in the diagram.  For $Y = (y_{i})_{i \geq 0}$ we denote $Y[n] = (y_{i} + n)_{i \geq 0}$. For two extended Young diagrams $Y=(y_i)_{i \geq 0}$ and $Y'=(y_i')_{i \geq 0}$ we say  $Y \subseteq Y'$ if $y_i \geq y_i'$ for all $i$ which means $Y$ is contained in $Y'$ as a diagram. Note that this containment is transitive.
   
The weight of a $k$-tuple of extended Young diagrams $\Y = (Y_1, Y_2, \ldots , Y_k)$ of charge $0$ is $ wt(\Y) = \sum_{i=1}^{k}wt(Y_{i})$.   Let $\mathcal{Y}(k\Lambda_0)$ denote the set of all $k$-tuples of extended Young diagrams of charge zero. The realization of the crystal $B(k\Lambda_{0})$ for $V(k\Lambda_0)$ is given in the following theorem.

  \begin{theorem}\label{JMMOTh} \cite{JMMO}  Let $V(k\Lambda_0)$ be the irreducible  $\widehat{sl}(n)$-module of highest weight $k\Lambda_0$ and let $B(k\Lambda_0)$ be its crystal.  Then $B(k\Lambda_0) = \{ \Y = (Y_{1}, \ldots , Y_{k}) \in \mathcal{Y}(k\Lambda_0) \mid Y_{1} \supseteq Y_{2} \supseteq \cdots \supseteq Y_{k}  \supseteq Y_1[n], \text{and for each } i \geq 0, \exists \  j \geq 1 \text{ s.t. } (Y_{j+1})_{i} > (Y_{j})_{i+1} \}$.
\end{theorem}

 \begin{remark}   Let $B(k\Lambda_0)_{k\Lambda_0 - \lambda_{A,B}^{\ell}}$ denote the set of $\Y \in B(k\Lambda_0)$ such that wt$(\Y) = k\Lambda_0 - \lambda_{A,B}^\ell$.  Then mult$_{k\Lambda_0}(k\Lambda_0 - \lambda_{A,B}^{\ell}) = |B(k\Lambda_0)_{k\Lambda_0 - \lambda_{A,B}^{\ell}}|$.
\end{remark}

\begin{example}\label{ex_as_crystal} Consider the weight $\mu = 7\Lambda_0 - \lambda_{((4,3,3),(3,2,2,1,1,1) )}^{10} = 7\Lambda_0 - 10 \alpha_0 - 7\alpha_1-5\alpha_2-3\alpha_3-2\alpha_4-\alpha_5-3\alpha_{10}-6\alpha_{11}$ in the $\widehat{sl}(12)$-module $V(7\Lambda_0)$.  
The tuple $Y = (Y_1, Y_2, Y_3, Y_4, Y_5, Y_6, Y_7)$ below is  an element of the crystal base $B(7\Lambda_0)$ of weight $7\Lambda_0 - \lambda_{A,B}^{10}$ with partitions $A = (4,3,3), B = (3,2,2,1,1,1) $. Here $\emptyset$ denotes the empty diagram of charge $0$. Indeed 
 $|B(7\Lambda_0)_{7\Lambda_0 - \lambda_{A,B}^{10}}| = 488 = {\rm mult}_{7\Lambda_0}(\mu)$.

$$Y =  \tiny {\left ( \tableau{0 & 1 & 2 & 3 & 4 & 5 \\ -1 & 0 & 1 & 2 & 3 & 4 \\ -2 & -1 & 0 & 1 }, \tableau{0 & 1 & 2 & 3 \\ -1 & 0 & 1 & 2 \\ -2}, \tableau{0 & 1 & 2 \\ -1 & 0  \\ -2 },
\tableau{0 & 1 \\ -1 }, \tableau{0 \\ -1} , \tableau{0}, \emptyset \right )}$$

\end{example}

\vskip 10pt

\section{Elements in  $B(k\Lambda_0)_{\lambda_{A,B}^{\ell}}$ and Partitions}

For $\ell \in \Z_{\geq 1}$, fix admissible pairs of partitions $(A =\{a_1 \geq a_2 \geq \cdots \geq a_r \geq 1\}$,  $B =\{b_1 \geq b_2 \geq \cdots \geq b_q \geq 1\})$ and the dominant maximal weight $\mu = k\Lambda_0 - \lambda_{A,B}^\ell$ where  $\ds \lambda_{A,B}^{\ell} = \ell \alpha_0 + \sum_{i=1}^{q-1} (\sum_{j=i+1}^q b_j)\alpha_i +  \sum_{i=1}^{r-1}(\sum_{j=i+1}^r a_j)\alpha_{n-i}$. In this section, we associate each element of the crystal $B(k\Lambda_0)_{k\Lambda_0 - \lambda_{A,B}^{\ell}}$ of weight $k\Lambda_0 - \lambda_{A,B}^{\ell}$ with a triple $(P,Q,\tau)$, where $P= (P_1, P_2, \cdots , P_{q-1})$ and $Q= (Q_1, Q_2, \cdots , Q_{r-1})$ are sequences of partitions and $\tau$ is a distinguished partition.  For a partition $\tau$ of integer $\ell \in \Z_{\geq 1}$  with at least $j$ parts we denote $S^{\tau}_j$ to be the set of partitions of $\ell-j$ such that exactly $j$ distinct parts of $\tau$ are decreased by 1 and $l(\tau)$ denote the length of $\tau$.  For example,  $S^{(3,2,2,1)}_2 = \{(2,2,1,1), (2,2,2), (3,1,1,1),  (3,2,1)\}$.

Define  $\mathcal{R}_{A,B}^{\ell}$ to be the set of triples $(P, Q, \tau)$ satisfying
\begin{enumerate}
\item $\tau \vdash \ell$, $l(\tau) \leq k$, $\tau_1 \leq \min\{r,q\}$.
\item $P = (P_1, P_2, \cdots , P_{r-1}), Q = (Q_1, Q_2, \cdots , Q_{q-1})$, {\rm where} \, $P_1 \in S_{a_1}^{\tau} , Q_1 \in S_{b_1}^{\tau}$,
$P_i \in S_{a_i}^{P_{i-1}}, 2 \leq i \leq r-1$, $Q_j \in S_{b_j}^{Q_{j-1}}, 2 \leq j \leq q-1$ and $P_{r-1} = (1, 1, \cdots , 1) \vdash a_r, Q_{q-1} =(1, 1, \cdots , 1) \vdash b_q$.
\end{enumerate}

\bigskip \bigskip

\begin{theorem}\label{card-R}  The multiplicity of $k\Lambda_0 - \lambda_{A,B}^\ell$ in $V(k\Lambda_0)$ is equal to $|\mathcal{R}_{A,B}^\ell|$. 
\end{theorem}

\begin{proof} 

It suffices to give a bijection between the sets $B(k\Lambda_0)_{k\Lambda_0 - \lambda_{A,B}^{\ell}}$ and $\mathcal{R}_{A,B}^\ell$. 
 
First, let  $Y = (Y_1, Y_2, \ldots, Y_k) \in B(k\Lambda_0)_{k\Lambda_0 - \lambda_{A,B}^{\ell}}$. We define  $\tau_j$ to be the number of 0-colored boxes, $p_{ij}$ to be the number of $(-i)$-colored boxes and $q_{mj}$ to be the $m$-colored boxes  in $Y_j$. Then by the defining properties of $Y$, we have $\tau =\{\tau_1 \geq \tau_2 \geq \cdots \geq \tau_{l(\tau)}\}\vdash \ell$ ,  and partitions $P_i = \{p_{i1} \geq p_{i2} \geq \cdots \geq p_{i,l(P_i)}\}$, $Q_m = \{q_{m1} \geq q_{m2} \geq \cdots \geq q_{m,l(Q_m)}\}$ with $l(\tau), l(P_i), l(Q_m) \leq k$ for  all $1\leq i \leq r-1$ , and $1\leq m \leq q-1$. Furthermore, $P_1 \in S_{a_1}^{\tau}$, $Q_1 \in S_{b_1}^{\tau}$, $P_i \in S_{a_i}^{P_{i-1}}$, $Q_j \in S_{b_j}^{Q_{j-1}}$, $2 \leq i \leq r-1, 2 \leq j \leq q-1$, and  neither $P_{r-1}$ nor $Q_{q-1}$ have any part larger than one.  We denote $P= (P_1, P_2, \cdots, P_{r-1})$ and $Q= (Q_1, Q_2, \cdots, Q_{q-1})$ as the sequences of partitions. Now, $(P,Q,\tau) \in \mathcal{R}_{A,B}^\ell$.

Conversely, let $(P,Q,\tau) \in \mathcal{R}_{A,B}^\ell$. We form a tuple of extended Young diagrams $Y = (Y_1, Y_2, \ldots, Y_k)$ as follows. For each $j$ from 1 to $k$, we place $\tau_j$ boxes of color 0, $(P_i)_j$ boxes of color $-i$ for $1 \leq i \leq r-1$, and $(Q_m)_j$ boxes of color $m$ for $1 \leq m \leq q -1$ in $Y_j$.   By the defining properties of $(P,Q,\tau)$, each $Y_i$ is an extended Young diagram, $Y_1 \supseteq Y_2 \supseteq \cdots \supseteq Y_k$, and $wt(Y) = k\Lambda_0 - \lambda_{A,B}^\ell$. Further, since the smallest color in $Y_1$ is $-(r-1)$ and $q+r \leq n$, we have $-(r-1) + n \geq q + 1 > 0$ and so $(Y_1[n])_i > 0$ for all $i$. Therefore, $Y_k \supseteq Y_1[n]$ and $(Y_{1}[n])_i > (Y_k)_{i+1}$ for all $i$. Thus,  $Y \in B(k\Lambda_0)_{k\Lambda_0 - \lambda_{A,B}^{\ell}}$.
\end{proof}

\begin{example}\label{ex_as_YZmu} Let $A =(4,3,3), B = (3,2,2,1,1,1), \ell = 10, Y = (Y_1, Y_2, \ldots, Y_7)$ as in Example \ref{ex_as_crystal}.  Counting the number of boxes of color 0 in each extended Young diagram in $Y$, we obtain $\tau = (3,2,2,1,1,1)$. Next, we count the boxes of color -1 and color -2 in each diagram and obtain $P_1 = (2,1,1,1,1)\in S_{4}^\tau $ and $P_2 = (1,1,1) \in  S_3^{P_1}$. Notice that $P_2 = (1,1,1) \vdash 3$, as required. Finally, we count the boxes of color 1, 2, 3, 4, and 5 in each diagram in $Y$, obtaining $Q_1 = (3,2,1,1) \in S_{3}^\tau, Q_2 = (2,2,1) \in S_2^{Q_1}, Q_3 = (2,1)\in S_2^{Q_2}, Q_4 = (2)\in S_1^{Q_3}, Q_5 = (1)\in S_1^{Q_4}$. Now, $(P = (P_1, P_2), Q = (Q_1, Q_2, Q_3, Q_4, Q_5), \tau) \in \mathcal{R}_{A,B}^\ell$. 
\end{example}

\section{Ordered Pairs of Standard Young Tableaux}

In this section, we show that the multiplicity of the weight $k\Lambda_0 - \lambda_{A,B}^\ell$ in $V(k\Lambda_0)$ is given by the the cardinality of a set of certain ordered pairs of standard Young tableaux. In order to prove this, we will use a sliding operation developed by Sch{\" u}tzenberger.  This procedure is often called the jeu de taquin, after a game with sliding pieces (c.f. \cite{F}). We begin with a skew Young tableau, say $\mu \setminus \lambda$.  A box in the deleted diagram $\lambda$ is called an inside corner (or a hole) if  the boxes immediately to its right and below are not in $\lambda$.   A box in $\mu$, but not in $\lambda$, is an outside corner if it does not have a box to its right or below.
 Starting with an inside corner, the procedure is to move the box to the right or below the hole with smaller entry to the hole, creating a hole in a new position. Then repeat the process by moving the box to the right or below with smaller entry to the hole until the hole is at an outside corner at which time we consider the hole to be removed. Once this sliding procedure is performed for each inside corner (or hole) the resulting diagram will be a tableau called the rectification of $\mu \setminus \lambda$.
The rectification of a skew diagram into a  tableau is reversible (c.f. \cite{F}), given that we know which outside corner each box was removed from. We begin by adding an empty box (hole) to the outside corner from which the last box was removed.  With a sequence of reverse slides, we move the hole to an inside corner.  Specifically,  in a reverse slide, the hole slides into the position of the box to its left or above whichever has the larger entry. The following examples illustrate rectification and its reverse. 

\begin{example}\label{forward_rect_ex}

In Figure \ref{forward_rect}, we begin with a skew diagram. We first consider the inside corner in the third row; the boxes to the right and below this hole are 10 and 5. Since $5 < 10$, we move the 5 into this hole. Next we move the 7 into the hole; then we move the 9; and then the hole is an outside corner and is removed. Now we consider the inside corner in the second row. We compare  5 and 6 and slide the 5 into the hole. Comparing 7 and 10, we move the 7 into the hole. In the next few steps, the hole slides down the first column, becomes an outside corner, and is removed. Finally, we consider the hole in the first row. We compare the 4 and 5 and slide the 4 into the hole. Now, we compare 6 and 8, and slide 6 into the hole. The hole moves down the second column, becomes an outside corner, and is removed. 

\begin{figure}[h]
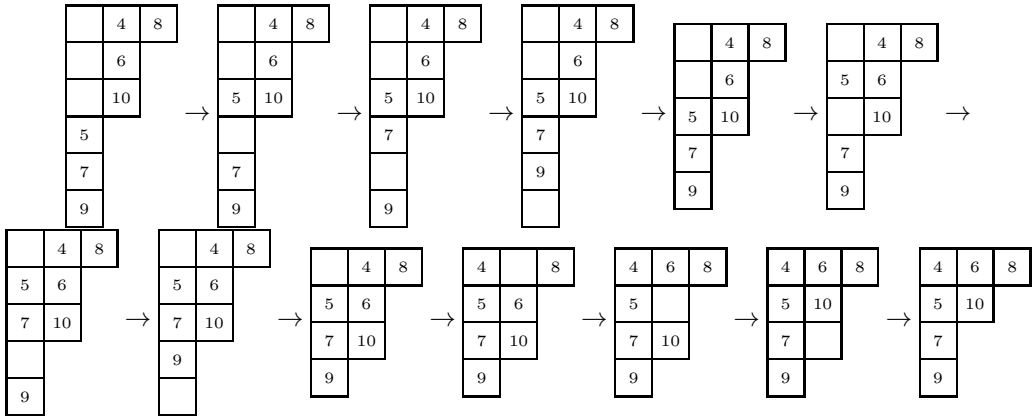


$\tiny{\tableau{ & 4 & 8 \\  & 6 \\  & 10 \\ 5 \\7 \\ 9} \to \tableau{ & 4 & 8 \\  & 6 \\ 5 & 10 \\  \\7 \\ 9} \to \tableau{ & 4 & 8 \\  & 6 \\ 5 & 10 \\7  \\ \\ 9}\to \tableau{ & 4 & 8 \\  & 6 \\ 5 & 10 \\7  \\ 9\\  \\} \to \tableau{ & 4 & 8 \\  & 6 \\ 5 & 10 \\7  \\ 9\\  }\to \tableau{ & 4 & 8 \\  5& 6 \\  & 10 \\7  \\ 9\\ } \to }$ $ \tiny{\tableau{ & 4 & 8 \\  5& 6 \\  7& 10 \\  \\ 9\\ }  \to \tableau{ & 4 & 8 \\  5& 6 \\  7& 10 \\  9\\ \\ }\to \tableau{ & 4 & 8 \\  5& 6 \\  7& 10 \\  9}\to \tableau{ 4&  & 8 \\  5& 6 \\  7& 10 \\  9} \to \tableau{ 4&  6& 8 \\  5&  \\  7& 10 \\  9} \to \tableau{ 4&  6& 8 \\  5&  10 \\  7& \\  9}  \to \tableau{ 4&  6& 8 \\  5&  10 \\  7 \\  9}}$
\caption{Example of rectification}
\label{forward_rect}
\end{figure}

\end{example}

\begin{example}\label{reverse_rect_ex} To show the reverse of the rectification process, we begin with a standard Young tableau, as in Figure \ref{rev_rectification}. For our purposes later in the paper, we add four to each entry. Suppose we know that we must add one box to each of the rows 1, 2, 3, and 6. We work from top to bottom, so we start by adding a hole to the end of the first row. Since there is no box above, we just move the hole to the left. We repeat until the hole is in the first column and has no boxes above it or to its left. Now, we add a hole to the second row. We compare 5 and 6 and since 6 is larger, we slide the 6 into the hole and now the hole is an inside corner. Next, we add an empty box to row 3. We compare the 6 and 7 and since $7 > 6$, we slide 7 into the hole. The hole is now an inside corner.  Finally, we add an empty box to row 6. There is only a box above it, so we slide the 10 into the hole and then the 9. At this point, we again have an inside corner. 

\begin{figure}[h]
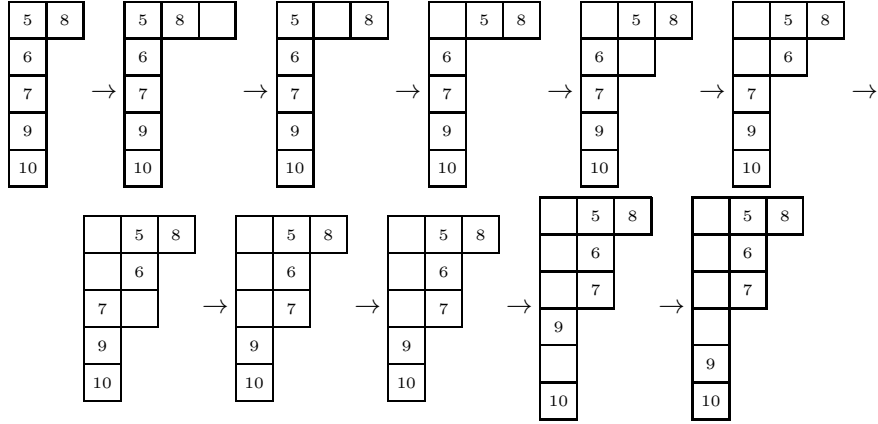


$\tiny{ \tableau{5 & 8 \\ 6 \\ 7 \\ 9 \\ 10} \to  \tableau{5 & 8 & \\ 6 \\ 7 \\ 9 \\ 10} \to  \tableau{5 &  & 8  \\ 6 \\ 7 \\ 9 \\ 10}  \to  \tableau{ & 5 & 8  \\ 6 \\ 7 \\ 9 \\ 10}  \to \tableau{ & 5 & 8  \\ 6 &  \\ 7 \\ 9 \\ 10}  \to \tableau{ & 5 & 8  \\  & 6 \\ 7 \\ 9 \\ 10} \to}$ 

\smallskip

$\tiny{ \tableau{ & 5 & 8  \\  & 6 \\ 7  & \\ 9 \\ 10} \to \tableau{ & 5 & 8  \\  & 6 \\ & 7 \\ 9 \\ 10} \to \tableau{ & 5 & 8  \\  & 6 \\ & 7 \\ 9 \\ 10 \\ } \to \tableau{ & 5 & 8  \\  & 6 \\ & 7 \\ 9 \\  \\ 10} \to \tableau{ & 5 & 8  \\  & 6 \\ & 7 \\  \\ 9  \\ 10}}$

\caption{Example of reverse of rectification}
\label{rev_rectification}
\end{figure}

\end{example}

\begin{lemma}\label{insertColumn}
Let $\tau$ be a standard Young tableau.  Using the reverse of Sch{\" u}tzenberger's sliding operation, if we insert a sequence of boxes in distinct rows, then in the resulting diagram, these boxes will appear in the first column in the same order.


\end{lemma}

\begin{proof}
Suppose we begin with a standard Young tableau, $\sigma$, with entries $\sigma_{i,j}$ in row $i$ and column $j$. We insert the first box at the end of row $t$ and after a series of reverse slides, it must become an inside corner in the first row of the first column. If $t >  1$, then one box from each of rows 1 through $t-1$ was moved down one row as the hole moved up to row one. Suppose the box that moved from the first row to the second was in column $c_1$. Then rows one and two of the diagram after the first hole is inserted are as in Figure \ref{rows1and2}. When we add the second hole to the diagram it must eventually enter row two. The reverse rectification process ensures that it does so to the left of column $c_1$. Since $\sigma_{2,i} > \sigma_{1,i}$ for $1 \leq i \leq c_1$, the hole must stay in row two and will become an inside corner in the first column. The case where $t=1$ is similar, but with no box moved down to row two. Analogously, it can be shown that each additional hole added to successively lower rows will be in the first column when it becomes an inside corner. 


%
%
%
%
%
%

\begin{figure}[h]

$\small{\tableau[l]{
\fr[] & \fr[] (1) &\fr[] (2) &\fr[] (3) &\fr[] (4) &\fr[] &\fr[] (c_1-1) &\fr[] (c_1) &\fr[] (c_1+1) &\fr[]  \\
 \fr[] (1) & & \sigma_{1,1} & \sigma_{1,2} & \sigma_{1,3}& \ldots & \sigma_{1,c_1 -2} & \sigma_{1,c_1-1} & \sigma_{1,c_1+1}  & \ldots \\\fr[] (2) &
\sigma_{2,1} & \sigma_{2,2} & \sigma_{2,3} & \sigma_{2,4} & \ldots & \sigma_{2,c_1-1} & \sigma_{1,c_1} & \sigma_{2,c_1} & \ldots } }$

\caption{Rows 1 and 2 of $\sigma$ after first hole inserted below row 1}
\label{rows1and2}
\end{figure}

\end{proof}

\begin{lemma}\label{sameinrev}
In a standard Young tableau, $\sigma$, where entry $a$ is in a row above the row with entry $a+1$, $a$ will remain in some row above the row with entry $a+1$ after any number of insertions using reverse slides. 
\end{lemma}

\begin{proof}
For the sake of contradiction, suppose that the $a$-box moves down to the row with the $(a+1)$-box, filling in a position immediately to the left of the $(a+1)$-box. This means that in the preceding step, the $a$-box was above and to the left of the $(a+1)$-box, which means that immediately above the $(a+1)$-box is a box with entry larger than $a$ and less than $a+1$, which cannot happen. 
\end{proof}

%
%
%
%
%
%
%
%
%
%
%
%
%
%
%

For fixed admissible pairs of partitions $(A =\{a_1 \geq a_2 \geq \cdots \geq a_r \} \vdash \ell$,  $B =\{b_1 \geq b_2 \geq \cdots \geq b_q \}\vdash \ell)$ we define sets  $U_1 = \{1, 2, \cdots , a_1\},
U_2 = \{a_1 + 1, a_1 + 2, \cdots , a_1 + a_2\}, \cdots , U_r = \{\sum_{i=1}^{r-1}a_i + 1, \sum_{i=1}^{r-1}a_i +2, \cdots , \sum_{i=1}^r a_i = \ell\}$ and $V_1 = \{1, 2, \cdots , b_1\},
V_2 = \{b_1 + 1, b_1 + 2, \cdots , b_1 + b_2\}, \cdots , V_q = \{\sum_{i=1}^{q-1}b_i + 1, \sum_{i=1}^{q-1}b_i +2, \cdots , \sum_{i=1}^q b_i = \ell\}$. Observe that $U_1\cup U_2\cup \cdots \cup U_r$ and $V_1\cup V_2\cup \cdots \cup V_q$ are both partitions of the set $\{1, 2, \cdots , \ell\}$.

We define $\mathcal{S}_{A,B}^\ell$ to be the set  of ordered pairs of standard Young tableaux $(M,N)$ satisfying the following conditions: 
\begin{enumerate}
\item $M$ and $N$ are standard Young tableaux of the same shape, each with $\ell$ boxes and at most  $k$ rows. Also the first row in each has at most $\min\{r,q\}$ number of boxes.
\item For each $j$, if $j_1, j_2 \in U_j \  ({\rm  resp.} \  V_j)$ and $j_1 < j_2$, then in $M$ (resp. $N$) $j_1$ must occur in a row above the row containing $j_2$.
\end{enumerate}
 
\begin{theorem}\label{pairsSYT} The multiplicity of $k\Lambda_0 - \lambda_{A,B}^\ell$ in $V(k\Lambda_0)$ is equal to $|\mathcal{S}_{A,B}^\ell|$. 
\end{theorem}

\begin{proof}
By Theorem \ref{card-R}, it suffices to show that there is a bijection between the sets $\mathcal{R}^\ell_{A,B}$ and $\mathcal{S}^\ell_{A,B}$. 

First, let $(P,Q,\tau) \in \mathcal{R}^\ell_{A,B}$. We use $P$ to create a standard Young tableau of shape $\tau$.  Recall that $P_{r-1}$ has largest part one and so we can uniquely form a standard Young tableau by placing the entries in $U_r$. We call this standard Young tableau $P_{r-1}^*$.    Next, we compare the shapes of $P_{r-2}$ and $P_{r-1}$.  Since $P_{r-1} \in S_{a_{r-1}}^{P_{r-2}}$, there are $a_{r-1}$ additional entries in $P_{r-2}$ in $a_{r-1}$ different rows, and we insert $a_{r-1}$ empty boxes into $P_{r-1}^*$ using the steps in the reverse sliding procedure.  Specifically, we insert the first box at the end of the first row in which $P_{r-2}$ is different than $P_{r-1}$, the second box in the next row that differs, and so on.  It follows from Lemma \ref{insertColumn} that these entries appear in the first column of the new diagram of shape $P_{r-2}$.  We fill in the entries in $U_{r-1}$ and call this standard Young tableau $P_{r-2}^*$.  We continue in this manner, creating a standard Young tableau $P_{r-3}^*$ of shape $P_{r-3}$ and so on until we create a standard Young tableau $P_1^*$ of shape $P_1$.  Finally, we create a standard Young tableau of shape $\tau$ by inserting $a_1$ boxes into $P_1^*$ in the appropriate rows.   Again, by Lemma \ref{insertColumn}, we have $a_1$ empty boxes in the first column.  We label these boxes 1 through $a_1$ and call the resulting standard Young tableau $M$.  When each $P_j^*$ is formed, the entries in the set $U_j$ are in its first column.  By Lemma \ref{sameinrev}, those entries remain in increasing rows even after more reverse slides occur.  Thus each of the sets $U_1, U_2, \cdots , U_r$ appear as entries in strictly increasing rows in $M$. Analogously, we use the tuples in $Q$ to construct a standard Young tableau $N$ of shape $\tau$ with the entries in each of the sets $V_1, V_2, \cdots , V_q$ appearing in increasing rows.  Then  $(M,N) \in \mathcal{S}_{A,B}^\ell$ by the construction.

Now, we let $(M,N) \in \mathcal{S}_{A,B}^\ell$.  To form $(P,Q,\tau) \in \mathcal{R}_{A,B}^\ell$, we first set $\tau$ to be the shape of $M$.  We note that since $1, 2, 3, \ldots, a_1$ appear in rows of  strictly increasing order in $M$, they must appear in the first column.  We remove these $a_1$ boxes and use Sch{\" u}tzenberger's sliding operation to form another standard Young tableau, which we call $P^*_1$ of shape $P_1$.  The process ensures that  each box  is removed from a row higher than the one from which a previous box was removed and so $P_1 \in S_{a_1}^\tau$.  Now, since $a_1 + 1, a_1 +2, \ldots, a_1 + a_2$ are the smallest entries in $P_1^*$ and appear in strictly increasing rows, these entries are in the first column of $P_1^*$.  We remove these boxes and use the sliding operation to remove each box from a different row.  The resulting standard Young tableau we call $P_2^*$ with shape $P_2$.  We continue in this way until we create the standard Young tableau $P_{r-1}^*$ with shape $P_{r-1} \in S_{a_{r-1}}^{P_{r-2}}$.  Now, set $P = (P_1, P_2, \ldots, P_{r-1})$ and similarly create $Q = (Q_1, Q_2, \ldots, Q_{q-1})$  from $N$. Then the triple $(P,Q,\tau) \in \mathcal{R}_{A,B}^\ell$ by construction, which proves the result. 
\end{proof}

\begin{example}\label{ex_as_MN} Continuing with the tuple $(P,Q,\tau)$ in Example \ref{ex_as_YZmu}, we form $M$, as in Figure \ref{makeMN}(a).  We start by placing the entries 8, 9, and 10 in a standard Young tableau of shape $P_2$.  Next, we compare $P_1$ and $P_2$ and we see that in $P_1$, there are 3 additional boxes, specifically in rows 1, 4, and 5.  We use reverse slides to insert a box into row 1, then row 4, and finally row 5.  At the conclusion of these reverse slides, we have three empty boxes in column 1 of a  tableau of shape $P_1$.  We label these boxes with 5, 6, and 7.  Finally, we compare the shapes of $\tau$ and $P_1$ to see that $\tau$ has 4 more boxes, in rows 1, 2, 3, and 6.  As in Example \ref{reverse_rect_ex}, we insert empty boxes into our standard Young tableau of shape $P_1$ in row 1, then row 2, then row 3, and finally row 6, culminating with a tableau of shape $\tau$ with four empty boxes in the first column.  We label these boxes 1 through 4 and call the result $M$.  Note that 5, 6, and 7 and also 8, 9, and 10 are still in strictly increasing rows.  To form $N$, as in Figure \ref{makeMN}(b), we follow a similar procedure, obtaining a sequence of standard Young tableaux of shapes $Q_5, Q_4, Q_3, Q_2, Q_1, \tau$.  Notice that in the first diagram in which each of the sets $V_1 = \{1,2,3\}, V_2 = \{4,5\}$, and $V_3 = \{6,7\}$ appears, all elements are in the first column.  In $N$, the elements of each of these sets are in strictly increasing rows.  For details on the map in the other direction, which uses forward slides, see the rectification of $N$ with entries 1,2, and 3 removed in Example \ref{forward_rect_ex}.

\begin{figure}[h!]
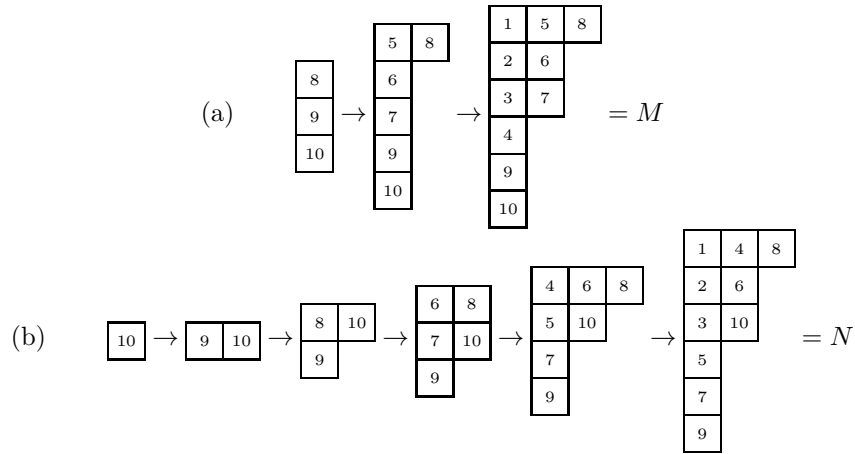


(a) \qquad $ \tiny{ \tableau{8 \\ 9 \\ 10} \rightarrow \tableau{5 & 8 \\ 6 \\ 7 \\ 9 \\10} \rightarrow \tableau{1 & 5 & 8 \\ 2 & 6 \\ 3 & 7 \\ 4 \\ 9 \\ 10} }= M$

(b) \qquad $ \tiny{ \tableau{10} \rightarrow \tableau{9 & 10} \rightarrow \tableau{8 & 10 \\ 9} \rightarrow  \tableau{6 & 8 \\ 7 & 10 \\ 9} \rightarrow \tableau{4 & 6 & 8 \\ 5& 10 \\ 7 \\ 9} \rightarrow \tableau{1 & 4 & 8 \\ 2 & 6 \\ 3 & 10 \\ 5 \\7 \\ 9} }= N$


\caption{Obtaining $(M,N)$}
\label{makeMN}
\end{figure}

\end{example}

\section{Pattern-Avoiding Permutations}

In this section we give the multiplicity of any maximal dominant weight $k\Lambda_0 - \lambda_{A,B}^\ell$ in $V(k\Lambda_0)$ as the number of certain pattern-avoiding permutations of $[\ell] = \{1, 2, \ldots, \ell \}$. For this we use the well-known RSK correspondence (c.f. \cite{Sta}), which associates each ordered pair $(M, N)$ of standard Young tableaux of same shape with $\ell$ boxes with certain pattern-avoiding permutation of $[\ell]$. One direction of the correspondence (call it the reverse algorithm) goes as follows. We begin with an ordered pair of standard tableaux $(M,N)$. We remove the largest box in $N$ and apply the reverse bump to the box in the corresponding location in $M$, say with entry $d$.  To do so, we place $d$ in the next highest row, replacing the largest entry less than $d$, say $e$, and  bumping that box up to the next highest row.  Now, we replace the largest entry in that row less than $e$ and bump that box up to the next highest row.  We continue this process until we remove a box from the first row.  The entry in this box becomes the last entry $w_\ell$ in the permutation of $[\ell]$ we are constructing.  We repeat with consecutively smaller boxes in $N$, filling in the words in the permutation $w$ from right to left.  This algorithm is known to be reversible.  In the following example we apply the RSK algorithm to construct a permutation of $[\ell]$ associated with a certain ordered pair of standard Young tableaux $(M, N)$.

\begin{example}  Using $(M,N)$ as in Example \ref{ex_as_MN},  we create a permutation (see Figure \ref{makeperm}).  First, we see that 10  appears in the second column of the third row of $N$.  Thus, we start the reverse bump process with the corresponding entry in $M$, which is 7.  We bump 7 to the next higher row, replacing the next entry smaller than it, which is a 6.  The 6, in turn bumps to the first row, replacing the 5, and so we have $w_{10} = 5$.  Next, we see that 9 is in the last row of the current $N$, so we bump the 10 in the corresponding position of the current $M$.  Continuing this process with each consecutively smaller box from $N$, we obtain the permutation  $w = 0947362815$ (representing 10 with 0). Notice that the entries in $U_1 = \{1,2,3,4\}$ appear in decreasing order in $w$, as do the entries in $U_2 = \{5,6,7\}$ and $U_3 = \{8, 9, 10\}$. Further, note that $w_1 > w_2 > w_3$, $w_4 > w_5$, and $w_6 > w_7$ and $V_1 = \{1,2,3\}$, $V_2 = \{4,5\}$ and $V_3 = \{6,7\}$. 

\begin{figure}[h]
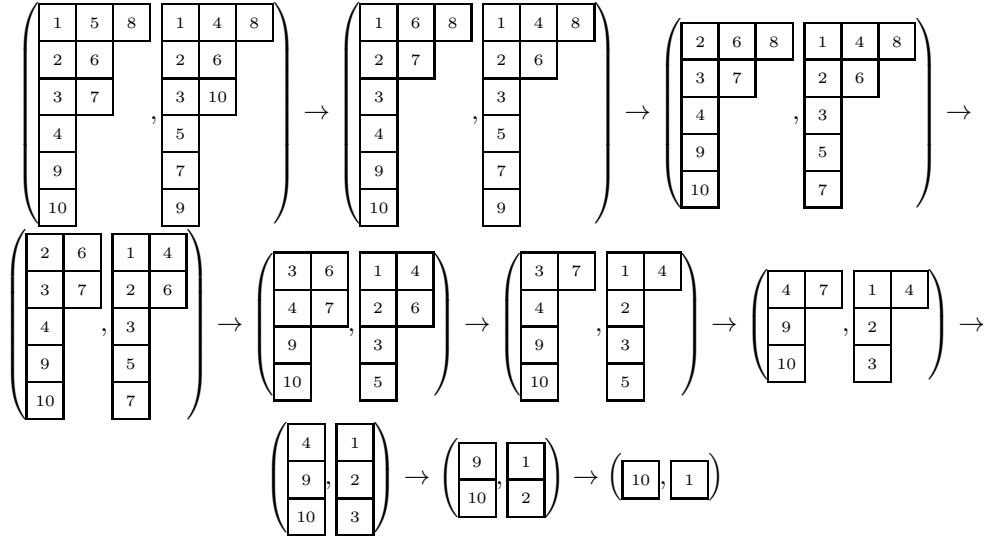


{ $\tiny{\left (\tableau{1 & 5 & 8 \\ 2 & 6 \\ 3 & 7 \\ 4 \\ 9 \\ 10}, \tableau{1 & 4 & 8 \\ 2 & 6 \\ 3 & 10 \\ 5 \\7 \\ 9} \right ) 
\to \left (\tableau{1 & 6 & 8 \\ 2 & 7 \\ 3 \\ 4 \\ 9 \\  10}, \tableau{1 & 4 & 8 \\ 2 & 6 \\ 3  \\ 5 \\7 \\ 9} \right ) \to \left (\tableau{2 & 6 & 8 \\ 3 & 7 \\ 4 \\ 9 \\  10}, \tableau{1 & 4 & 8 \\ 2 & 6 \\ 3  \\ 5 \\7 } \right ) \to }$ }

{$\tiny{\left (\tableau{2 & 6  \\ 3 & 7 \\ 4 \\ 9 \\  10}, \tableau{1 & 4 \\ 2 & 6 \\ 3  \\ 5 \\7 }  \right )  
\to \left (\tableau{3 & 6  \\ 4 & 7 \\  9 \\  10},\tableau{1 & 4  \\ 2 & 6 \\ 3  \\ 5  } \right )\to \left (\tableau{3 & 7  \\ 4  \\  9 \\  10}, \tableau{1 & 4 \\ 2  \\ 3  \\ 5  }  \right )  \to  \left (\tableau{4 & 7  \\ 9  \\  10}, \tableau{1 & 4 \\ 2 \\ 3   }  \right ) \to }$ }

{$\tiny{ \left (\tableau{4  \\ 9  \\  10}, \tableau{1   \\ 2  \\ 3  } \right ) \to  \left (\tableau{ 9 \\  10}, \tableau{1   \\ 2   } \right )  \to  \left (\tableau{ 10}, \tableau{1     } \right ) }$
}
\caption{RSK correspondence applied to $(M,N)$}
\label{makeperm}
\end{figure}

\end{example}

Recall that  a $(k+1)k\ldots21$-avoiding permutation of $[\ell]$ is a permutation with no decreasing subsequence of length $k+1$.  It is known that  using the RSK correspondence (c.f. \cite{Sta}, Corollary 7.23.12)  the number of ordered pairs of standard Young tableaux with the same shape $\mu \vdash \ell$ and less than or equal to $k$ rows is the same as the number of $(k+1)k(k-1)\cdots21$-avoiding permutations of $[\ell]$.

To prove the main result of this section, we need the following consequences of the RSK correspondence.

\begin{lemma}\label{Mlemma}
Suppose the permutation $w$ corresponds with the ordered pair of standard Young tableaux $(M,N)$ through the RSK correspondence.  Then $a$ appears to the right of $a-1$ in $w$ if and only if $a$ is in a row below $a-1$ in $M$. 
\end{lemma}

\begin{proof}
Suppose that $a$ appears to the left of $a-1$ in $w$. When constructing $(M,N)$ from $w$ via the RSK correspondence, $a$ will already have been inserted into $M$ when $a-1$ is inserted. So either $a-1$ will remain above the row containing $a$, or, if it enters the row with $a$, it will bump $a$ to the next row of $M$. 

Now, suppose that in $M$, $a$ is in a row below $a-1$. We reverse the RSK algorithm to create $w$. At any move in which $a$ is reverse bumped to the row containing $a-1$, $a-1$ will be the greatest entry less than $a$ and thus be reverse bumped up.  As such, $a-1$ will be removed from $M$ to be placed in $w$ prior to $a$.  
\end{proof}

%
%
%

\begin{lemma}\label{Nlemma}
Suppose the permutation $w = w_1w_2 \cdots w_{\ell}$ and the ordered pair of standard Young tableau $(M,N)$ correspond with each other through the RSK correspondence.  Then $w_a> w_{a+1}$ if and only if the box with entry $a$ appears in a row above the box with entry $a+1$ in $N$. 
\end{lemma}

\begin{proof}
It follows from Lemma 7.23.1 in \cite{Sta} that if $w_a > w_{a+1}$, then the box with entry $a$ appears in a row above the box with entry $a+1$ in $N$. 

Now suppose that the entry $a$ is above the entry $a+1$ in $N$ and use the reverse of the RSK algorithm to form $w$. Specifically, consider the step where we remove the box with entry $a+1$ in row $t$ of $N$. Then we  reverse bump the corresponding box, say $x_t$, in $M$.  A series of reverse bumps occurs with entries $x_t > x_{t-1} > \cdots > x_2$ each being bumped up one row until $x_2$ bumps some $x_1 < x_2$ out of the diagram and we set $w_{a+1} = x_1$.  Now, suppose the box with entry $a$ is at the end of row $d < t$ in $N$ with corresponding box $y_d$ in $M$.  Since $y_d$ is at the end of a row containing $x_{d+1}$,  $y_d \geq x_{d+1} > x_d$.  When $y_d$ is bumped to row $d-1$ (which contains $x_d$), it bumps some $y_{d-1}$ where $x_d \leq y_{d-1} < y_d$.  Continuing in this manner at each row $i$ we bump some $y_i \geq x_{i+1}$ to the next higher row.  Finally, we bump $y_1 \geq x_2$ from the diagram to set $w_a = y_1$.  Since $y_1 \geq x_2 > x_1 = w_{a+1}$, we have $w_a > w_{a+1}$.  
\end{proof}

\begin{theorem} The multiplicity of  $k\Lambda_0 - \lambda_{A,B}^\ell$ of $V(k\Lambda_0)$ is the number of all $(k+1)k(k-1)\cdots 21$-avoiding permutations $w$ of $[\ell]$ such that
\begin{itemize}
\item  For each $j$, if $j_1, j_2 \in U_j$ with $w_{a_1} = j_1 < w_{a_2} = j_2$, then $a_1 >   a_2$. 
\item For each $j$, if $j_1, j_2 \in V_j$ with $j_1 < j_2$, then $w_{j_1} > w_{j_2}$.
\end{itemize}

\end{theorem}

\begin{proof}
Since for any $i,j$, entries in $U_i$ and $V_j$ are consecutive integers, it follows from Lemmas \ref{Mlemma} and \ref{Nlemma} that for given $\ell$, $A$, $B$, the set of permutations in the statement of the theorem is bijective with the set of ordered pairs of Young tableaux $(M, N) \in \mathcal{S}_{A,B}^{\ell}$. Hence by Theorem \ref{pairsSYT}, the result follows. 


%
\end{proof}


\begin{thebibliography}{}


\bibitem{F} Fulton, William: Young Tableaux:  With Applications to Representation Theory and Geometry.  Cambridge University Press, New York (1997).  


\bibitem{JM1} Jayne, R.L., Misra, K.C.: On multiplicities of maximal dominant weights of $\widehat{sl}(n)$-modules, {\it Algebr. Represent. Th.}, {\bf 17}, 1303--1321 (2014).

\bibitem{JM2} Jayne, R.L., Misra, K.C.: Lattice Paths, Young Tableaux, and Weight Multiplicities, {\it Annals of Combinatorics}, {\bf 22}, 147--156 (2018).  

\bibitem{JM3} Jayne, R.L., Misra, K.C.: Multiplicities of some maximal dominant weights of the $\widehat{s\ell}(n)$-modules $V(k\Lambda_0)$, {\it Algebr. Represent. Th.}, to appear. 

\bibitem{JMMO} Jimbo, M., Misra, K.C.,  Miwa, T., Okado, M.: Combinatorics of representations of $U_{q}\left (\widehat{sl}(n) \right )$ at $q=0$,  {\it Commun. in Math. Phys.}, {\bf 136} (1991) 543--566.

\bibitem{Kac} Kac, V. G.: Infinite Dimensional Lie Algebras, 3rd edn. Cambridge University Press, New York (1990).

\bibitem{KKO} Kim, Jang Soo, Lee, Kyu-Hwan, and Oh, Se-jin: Dominant maximal weights of highest weight modules and Young tableaux, {\it Sem. Lothar. Combin.} {\bf 78B} (2017), Art. 25, (12pp).

\bibitem{KOO} Kim, Young-Hun, Oh, Se-jin, and Oh, Young-Tak: Cyclic sieving phenomenon on dominant maximal weights over affine Kac-Moody algebras, {\it Adv. Math} {\bf 374} (2020), 107336, (75 pp).

\bibitem{RSW} Reiner, V., Stanton, D., and White, D.: The cyclic sieving phenomenon, {\it J. Combin. Theory Ser. A} {\bf 108} (2004), 17--50.

\bibitem{Sc} Schensted, C.:  Longest increasing and decreasing subsequences, {\it Canad. J. Math.} {\bf 13} 179--191 (1961).  

\bibitem{Sta} Stanley, R.P.: Enumerative Combinatorics. Vol. 2. Cambridge University Press, New York (1999).

\bibitem{Tsu} Tsuchioka, S.: Catalan numbers and level 2 weight structures of $A^{(1)}_{p-1}$, {\it RIMS K\v{o}ky\v{u}roku Bessatsu}. {\bf B11} 145-154 (2009).

\bibitem{TW} Tsuchioka, S., Watanabe, M.:  Pattern avoidance seen in multiplicities of maximal weights of affine Lie algebra representations, {\it Proc. Amer. Math. Soc.}, {\bf 146} 15--28 (2018). 

\end{thebibliography}
\end{document}